\documentclass[12pt,twoside]{amsart}
\usepackage{amsmath}
\usepackage{amsthm}
\usepackage{amsfonts}
\usepackage{amssymb}
\usepackage{latexsym}
\usepackage{mathrsfs}
\usepackage{amsmath}
\usepackage{amsthm}
\usepackage{amsfonts}
\usepackage{amssymb}
\usepackage{latexsym}
\usepackage{geometry}
\usepackage{dsfont}
\usepackage[dvips]{graphicx}
\usepackage{color}
\usepackage[all]{xy}

\date{}
\allowdisplaybreaks[4] \footskip=15pt
\renewcommand{\uppercasenonmath}[1]{}

\usepackage{graphicx,amssymb}
\usepackage[all]{xy}
\usepackage{amsmath}

\allowdisplaybreaks
\usepackage{amsthm}
\usepackage{color}

\theoremstyle{plain}
\newtheorem{theorem}{Theorem}[section]
\newtheorem{proposition}[theorem]{Proposition}
\newtheorem{lemma}[theorem]{Lemma}
\newtheorem{corollary}[theorem]{Corollary}
\newtheorem{example}[theorem]{Example}
\newtheorem*{open question}{Open Question}
\newtheorem{definition}[theorem]{Definition}

\theoremstyle{definition}
\newtheorem*{acknowledgement}{Acknowledgement}

\theoremstyle{remark}
\newtheorem{remark}[theorem]{Remark}
\newcommand{\C}{\mathcal{C}}

\newcommand{\Tor}{\mbox{\rm Tor}}

\newcommand{\Id}{\mathrm{Id}}

\def\p{\frak p}
\def\m{\frak m}

\def\Hom{{\rm Hom}}
\def\Ext{{\rm Ext}}
\def\Tor{{\rm Tor}}

\def\Ker{{\rm Ker}}

\def\Im{{\rm Im}}
\def\Coker{{\rm Coker}}

\def\Max{{\rm Max}}

\def\Spec{{\rm Spec}}
\def\Max{{\rm Max}}
\def\Id{{\rm Id}}
\begin{document}
\begin{center}
{\large  \bf Characterizing $S$-projective modules and $S$-semisimple rings by uniformity}

\vspace{0.5cm} Xiaolei Zhang$^{a}$,\ Wei Qi$^{a}$
%\bigskip

{\footnotesize $a$.\ School of Mathematics and Statistics, Shandong University of Technology, Zibo 255049, China\\

Corresponding author: Wei Qi, E-mail: qwrghj@126.com\\}
\end{center}
%\begin{figure}[b]
%\rule[-2.5truemm]{5cm}{0.1truemm}\\[2mm]
%{\small }
%\end{figure}

%\begin{figure}[b]
%\rule[-2.5truemm]{5cm}{0.1truemm}\\[2mm]
%{\small }
%\end{figure}
\bigskip
\centerline { \bf  Abstract}
\bigskip
\leftskip10truemm \rightskip10truemm \noindent

Let $R$ be a ring and $S$  a multiplicative subset of $R$.  An $R$-module $P$ is called  uniformly $S$-projective provided that the induced sequence $0\rightarrow \Hom_R(P,A)\rightarrow \Hom_R(P,B)\rightarrow \Hom_R(P,C)\rightarrow 0$ is $u$-$S$-exact for any $u$-$S$-short exact sequence $0\rightarrow A\rightarrow B\rightarrow C\rightarrow 0$. Some characterizations and properties of  $u$-$S$-projective modules are obtained.   The notion of $u$-$S$-semisimple modules is also introduced. A ring $R$ is called a $u$-$S$-semisimple ring provided that any free  $R$-module is $u$-$S$-semisimple. Several characterizations of $u$-$S$-semisimple rings are provided in terms of  $u$-$S$-semisimple modules, $u$-$S$-projective modules, $u$-$S$-injective modules and $u$-$S$-split $u$-$S$-exact sequences.
\vbox to 0.3cm{}

\leftskip0truemm \rightskip0truemm
\bigskip
%\section { \bf Introduction    }
%\bigskip

\section{Introduction}
Throughout this article, $R$ is always  a commutative ring with identity and $S$ is always a multiplicative subset of $R$, that is, $1\in S$ and $s_1s_2\in S$ for any $s_1\in S$ and any $s_2\in S$.  Let $S$  be  multiplicative subset of $R$. In 2002, Anderson and Dumitrescu \cite{ad02} introduced the notion of an \emph{$S$-Noetherian ring} $R$, that is, for any ideal $I$ of $R$ there exists a finitely generated sub-ideal $K$ of $I$ and an element $s\in S$ such that $sI\subseteq K$, i.e., $I/K$ is $u$-$S$-torsion by \cite{z21}. Since then,  $S$-analogues of other well-known classes of rings such as Artinian rings, coherent rings, almost perfect rings, GCD domains and strong Mori domains, were introduced and studied extensively in \cite{dhz19,s19,bh18,l15,lO14,kkl14,S20}.

Now let's go back to the definition of $S$-Noetherian rings. Notice that the element $s\in S$ such that $sI\subseteq K$ is decided by the ideal $I$ for $S$-Noetherian rings. This situation makes it difficult to characterize $S$-Noetherian rings  from the perspective of module-theoretic viewpoint. In order to overcome this difficulty, Qi and Kim et al. \cite{QKWCZ21} recently introduced the notion of \emph{uniformly $S$-Noetherian rings} $R$ for which there exists an element $s\in S$  such that for any ideal $I$ of $R$, $sI \subseteq K$ for some finitely generated sub-ideal $K$ of $I$. They also introduced the notion of  $u$-$S$-injective modules and finally showed that a ring  $R$ is  uniformly $S$-Noetherian if and only if  any direct sum of injective modules is $u$-$S$-injective in the case that $S$ is a regular multiplicative set. Another ``uniform'' case is that of uniformly $S$-von Neumann regular rings introduced by the first author of this paper (see \cite{z21}). The author in \cite{z21} first introduced $u$-$S$-flat modules using $u$-$S$-torsion modules, and then gave the notion of uniformly $S$-von Neumann regular rings extending von Neumann regular rings with uniformity on the multiplicative subset $S$. Finally, he characterized  uniformly $S$-von Neumann regular rings by using  $u$-$S$-flat modules.

The main motivation of this paper is to introduce and study  the uniformly $S$-versions of projective modules and semisimple rings. In Section 2 of this article, we first introduce the notion of $u$-$S$-split $u$-$S$-exact sequence (see Definition \ref{s-split-s-exact-sequence}). Dualizing the $u$-$S$-injective modules, we introduce the notion of $u$-$S$-projective module and show that an $R$-module $P$ is $u$-$S$-projective if and only if $\Ext_R^1(P,M)$ is $u$-$S$-torsion for any $R$-module $M$, if and only if any $u$-$S$-exact sequence ending at $P$ is $u$-$S$-split (see Theorem \ref{s-proj-char}). We also give a new local characterization of projective modules in Proposition \ref{s-flat-loc-char}. In Section 3 of this article, we first give the notion of a $u$-$S$-semisimple module $M$, that is, any $u$-$S$-shortly exact sequence with middle term $M$ is $u$-$S$-split. And then we introduced  the notion of $u$-$S$-semisimple ring over which every free module is $u$-$S$-semisimple.  We prove that a ring $R$ is a $u$-$S$-semisimple ring if and only if every $R$-module is $u$-$S$-semisimple,  if and only if every $u$-$S$-short exact sequence is $u$-$S$-split,  if and only if every $R$-module is $u$-$S$-projective,  if and only if every $R$-module is $u$-$S$-injective (see Theorem \ref{s-semisimple-char-}). By Corollary \ref{s-semisimple-char},  a $u$-$S$-semisimple ring is both uniformly $S$-Noetherian and  uniformly $S$-von Neumann regular.   We also show that if $S$  is  a regular multiplicative subset of $R$, then $R$ is a $u$-$S$-semisimple ring if and only if  $R$ is  semisimple (see Proposition \ref{s-vn-vn}). A non-trivial example of a $u$-$S$-semisimple ring which is not semisimple is given in Example \ref{us-sim-not-sem}. Finally, we give a new characterization of semisimple rings (see  Proposition \ref{s-semi-loc-char}).

\section{$u$-$S$-split $u$-$S$-exact sequences and $u$-$S$-projective modules}

Let $R$ be a ring and $S$ a multiplicative subset of $R$. Recall from  \cite[Definition 1.6.10]{fk16} that  an $R$-module $T$ is called a $u$-$S$-torsion module  provided that there exists an element $s\in S$ such that $sT=0$. Suppose $M$, $N$ and $L$ are $R$-modules.
\begin{enumerate}
\item An  $R$-homomorphism $f:M\rightarrow N$ is called a $u$-$S$-monomorphism $($resp.,    $u$-$S$-epimorphism$)$ provided that $\Ker(f)$ $($resp., $\Coker(f))$ is a  $u$-$S$-torsion module.

 \item    An  $R$-homomorphism $f:M\rightarrow N$ is  called a $u$-$S$-isomorphism  provided that $f$ is both a $u$-$S$-monomorphism and a $u$-$S$-epimorphism.

\item An $R$-sequence  $M\xrightarrow{f} N\xrightarrow{g} L$ is called  $u$-$S$-exact provided that there is an element $s\in S$ such that $s\Ker(g)\subseteq \Im(f)$ and $s\Im(f)\subseteq \Ker(g)$.

\item    A $u$-$S$-exact sequence $0\rightarrow A\rightarrow B\rightarrow C\rightarrow 0$ is called a short $u$-$S$-exact sequence.
\end{enumerate}
It is easy to verify that  $f:M\rightarrow N$ is a $u$-$S$-monomorphism  $($resp.,   $u$-$S$-epimorphism$)$ if and only if  $0\rightarrow M\xrightarrow{f} N$   $($resp., $M\xrightarrow{f} N\rightarrow 0$ $)$ is $u$-$S$-exact.

Suppose $M$ and $N$ are $R$-modules. We say $M$ is $u$-$S$-isomorphic to $N$ if there exists a $u$-$S$-isomorphism $f:M\rightarrow N$. A family $\C$  of $R$-modules  is said to be closed under $u$-$S$-isomorphisms if $M$ is $u$-$S$-isomorphic to $N$ and $M$ is in $\C$, then $N$ is  also in  $\C$.  One can deduce from the following Lemma \ref{l} that the $u$-$S$-isomorphism is actually an equivalence relation.

\begin{lemma}\label{l}
Let $R$ be a ring and  $S$ a multiplicative subset of $R$. Let $M$ and $N$ be $R$-modules. Suppose there is a $u$-$S$-isomorphism $f:M\rightarrow N$. Then there is a $u$-$S$-isomorphism $g:N\rightarrow M$ and $t\in S$ such that $f\circ g=t\Id_N$ and $g\circ f=t\Id_M$.
\end{lemma}
\begin{proof} Let $f:M\rightarrow N$ be a $u$-$S$-isomorphism. Then there is $s\in S$ such that $sN\subseteq \Im(f)$ and $s\Ker(f)=0$. For $y\in N$ pick $x \in M$ with $f(x)=sy$ and define $g(y)=sx$. Suppose $y\in N$ and pick $x_1,x_2\in M$ such that $f(x_1)=sy=f(x_2)$. Then $x_1-x_2\in\Ker(f)$. So $sx_1=s_2x_2$. Thus $g$ is well-defined. One can check that  $g$ is also linear. Trivially, $g$ is a $u$-$S$-isomorphism with  $f\circ g=s^2\Id_N$ and $g\circ f=s^2\Id_M$.
%Consider the following commutative diagram:
%$$\xymatrix{
%0\ar[r]^{}& \Ker(f)\ar[r]^{} &M\ar[rr]^{f}\ar@{->>}[rd] &&N \ar[r]^{} & \Coker(f)\ar[r]^{} &  0\\
%  & & &\Im(f) \ar@{^{(}->}[ru] &&  &   \\}$$
%with $s\Ker(f)=0$ and $sN\subseteq \Im(f)$ for some $s\in S$. Define $g_1:N\rightarrow \Im(f)$ where $g_1(n)=sn$ for any $n\in N$. Then $g_1$ is a well-defined $R$-homomorphism since $sn\in \Im(f)$. Define $g_2:\Im(f)\rightarrow M$ where $g_2(f(m))=sm$. Then  $g_2$ is  well-defined $R$-homomorphism. Indeed, if $f(m)=0$, then $m\in \Ker(f)$ and so $sm=0$. Set $g=g_2\circ g_1:N\rightarrow M$. We claim that $g$ is a $u$-$S$-isomorphism. Indeed, let $n$ be an element in $\Ker(g)$. Then $sn=g_1(n)\in \Ker(g_2)$. Note that $s\Ker(g_2)=0$. Thus $s^2n=0$. So $s^2\Ker(g)=0$. On the other hand, let $m\in M$. Then $g(f(m))=g_2\circ g_1(f(m))=g_2(f(sm))=s^2m$. Set $t=s^2\in S$. Then $g\circ f=t\Id_M$ and $tm\in \Im(g)$. So $tM\subseteq \Im(g)$. It follows that $g$ is  a $u$-$S$-isomorphism. It is also easy to verify that  $f\circ g=t\Id_N$.
\end{proof}

\begin{remark}
Let $R$ be a ring,  $S$ a multiplicative subset of $R$ and $M$ and $N$ $R$-modules. Then the condition ``there is an $R$-homomorphism $f:M\rightarrow N$  such that $f_S:M_S\rightarrow N_S$ is an isomorphism''  does not mean ``there is an $R$-homomorphism $g:N\rightarrow M$  such that $g_S:N_S\rightarrow M_S$ is an isomorphism''.

Indeed, let $R=\mathbb{Z}$ be the ring of integers, $S=R-\{0\}$  and $\mathbb{Q}$ the quotient field of integers. Then the embedding map $f:\mathbb{Z}\hookrightarrow \mathbb{Q}$ satisfies $f_S:\mathbb{Q}\rightarrow \mathbb{Q}$ is an isomorphism. However, since $\Hom_\mathbb{Z}(\mathbb{Q},\mathbb{Z})=0$, there does not exist any $R$-homomorphism $g:\mathbb{Q}\rightarrow \mathbb{Z}$  such that $g_S:\mathbb{Q}\rightarrow \mathbb{Q}$ is an isomorphism.
\end{remark}

Recall that an  exact sequence $0\rightarrow A\xrightarrow{f} B\xrightarrow{g} C\rightarrow 0$ is split provided that there is an $R$-homomorphism $f': B\rightarrow A$  such that $f'\circ f=\Id_A$.

\begin{definition}\label{s-split-s-exact-sequence}
Let $\xi: 0\rightarrow A\xrightarrow{f} B\xrightarrow{g} C\rightarrow 0$ be a $u$-$S$-short exact sequence. Then $\xi$ is said to be $u$-$S$-split (with respect to $s$) provided that there is  $s\in S$ and $R$-homomorphism $f':B\rightarrow A$ such that $f'(f(a))=sa$ for any $a\in A$, that is, $f'\circ f=s\Id_A$.
\end{definition}

Obviously, any split exact sequence is $u$-$S$-split. Certainly, an  exact sequence $0\rightarrow A\xrightarrow{f} B\xrightarrow{g} C\rightarrow 0$ splits if and only if there exists an $R$-homomorphism $g':C\rightarrow B$ such that $g\circ g'=\Id_C$. Now, we give the uniformly $S$-version of this result.

\begin{lemma}\label{s-split} Let $R$ be a ring and $S$ a multiplicative subset of $R$. A $u$-$S$-short exact sequence $\xi: 0\rightarrow A\xrightarrow{f} B\xrightarrow{g} C\rightarrow 0$ is $u$-$S$-split if and only if  there is  $s\in S$ and $R$-homomorphism $g':C\rightarrow B$ such that $g(g'(c))=sc$ for any $c\in C$, that is, $g\circ g'=s\Id_C$ for some $s\in S$.
\end{lemma}
\begin{proof} Choose the $R$-homomorphism $f': B\rightarrow  A$ and $s'\in S$ such that  $f'\circ  f=s'\Id_A, s'C\subseteq  \Im(g)$, $s'\Ker(g)\subseteq \Im(f)$ and  $s'\Im(f)\subseteq \Ker(g)$.
Define the map $g': C\rightarrow  B$ as follows.
For $z\in C$, pick $y\in B$ with $g(y) = s'z$ and define  $g'(z) = s'^2y -s'f(f'(y))$.
When  $g(y) = g(y') = s'z$, we pick $x\in A$ with $f(x) = s'(y-y')$, so
   $(s'^2y - s'f(f'(y))) - (s'^2y' - s'f(f'(y'))) = s'^2(y-y') - s'f(f'(y-y'))
   = s'f(x) - f(f'(f(x))) = s'f(x) - s'f(x) = 0$, thus $g'$ is well-defined.
It also can be checked that $g'$ is linear. Finally, if  $g(y) = s'z$,  we have
   $g(g'(z)) = g(s'^2y - s'f(f'(y))) = s'^3z$ because $s'g\circ f=0$.  Setting $s=s'^3$, we have $g\circ g' = s\Id_C$. The sufficiency can be proved similarly.

\end{proof}

Recall from \cite[Definition 4.1]{QKWCZ21} that an $R$-module $E$ is called $u$-$S$-injective provided that the induced sequence
 $$0\rightarrow \Hom_R(C,E)\rightarrow \Hom_R(B,E)\rightarrow \Hom_R(A,E)\rightarrow 0$$
 is $u$-$S$-exact for any $u$-$S$-exact sequence $0\rightarrow A\rightarrow B\rightarrow C\rightarrow 0$.
 By \cite[Theorem 4.3]{QKWCZ21}, an $R$-module $E$  is $u$-$S$-injective, if and only if for any short exact sequence $0\rightarrow A\xrightarrow{f} B\xrightarrow{g} C\rightarrow 0$, the induced sequence $0\rightarrow \Hom_R(C,E)\xrightarrow{g^\ast} \Hom_R(B,E)\xrightarrow{f^\ast} \Hom_R(A,E)\rightarrow 0$ is  $u$-$S$-exact, if and only if $\Ext_R^1(M,E)$ is  $u$-$S$-torsion for any $R$-module $M$, if and only if $\Ext_R^n(M,E)$ is  $u$-$S$-torsion for any $R$-module $M$ and any $n\geq 1$.
We can characterize  $u$-$S$-injective modules using $u$-$S$-exact sequences.
\begin{proposition}\label{s-inj-exac} Let $R$ be a ring,  $S$ a multiplicative subset of $R$  and $E$ an $R$-module. Then the following statements are equivalent:
\begin{enumerate}
\item $E$ is $u$-$S$-injective;
\item for any $u$-$S$-monomorphism $A\xrightarrow{f} B$  there exists $s\in S$ such that for any $R$-homomorphism $h: A\rightarrow E$, there exists an $R$-homomorphism $g: B\rightarrow E$ satisfying $sh=g\circ f$;
\end{enumerate}
\end{proposition}
\begin{proof}  $(1)\Rightarrow (2)$: Set $C = \Coker(f)$. As $0\rightarrow A \xrightarrow{f} B\rightarrow C \rightarrow 0$ is $u$-$S$-exact, we get the $u$-$S$-epimorphism  $f^*: \Hom(B,E)\rightarrow \Hom(A,E)$.
Pick $s\in S$ with $s\Hom(A,E)\subseteq \Im(f^*)$. Then $sh=g\circ f$ for some linear map $g:B \rightarrow E.$

$(2)\Rightarrow (1)$: Let $M$ be an $R$-module and $0\rightarrow N\xrightarrow{i} P\rightarrow M\rightarrow 0$  a short exact sequence of $R$-modules with $P$ projective. Then we have a long exact sequence $0\rightarrow \Hom_R(M,E)\rightarrow \Hom_R(P,E)\xrightarrow{i^\ast} \Hom_R(N,E)\rightarrow \Ext_R^1(M,E)\rightarrow 0$. By $(2)$, $i^\ast$ is a $u$-$S$-epimorphism. Thus  $\Ext_R^1(M,E)$ is $u$-$S$-torsion. So $E$ is $u$-$S$-injective.
\end{proof}

\begin{lemma} The following two statements are equivalent:
\begin{enumerate}
\item   any $u$-$S$-short exact sequence  $0\rightarrow E\rightarrow B\rightarrow C\rightarrow 0$ beginning at $E$ is $u$-$S$-split ;
\item  any short exact sequence  $0\rightarrow E\rightarrow B\rightarrow C\rightarrow 0$ beginning at $E$ is $u$-$S$-split.
\end{enumerate}
\end{lemma}
\begin{proof}

$(1)\Rightarrow (2)$: Obvious.

$(2)\Rightarrow (1)$: Let  $0\rightarrow E\xrightarrow{f}  B\xrightarrow{g} C\rightarrow 0$ be a  $u$-$S$-short exact sequence.  Then  there is a short exact sequence $0\rightarrow \Im(f)\rightarrow B\rightarrow \Coker(f)\rightarrow 0$. If we denote by $f':E\twoheadrightarrow\Im(f)$ to be the natural epimorphism, then there is an $u$-$S$-isomorphism $h:\Im(f)\rightarrow E$ such that $h\circ f'=s_1\Id_E$ for some $s_1\in S$ by Lemma \ref{l}. Consider the following push-out:
  $$\xymatrix@R=20pt@C=25pt{
0\ar[r]^{}& E \ar[r]^{v}&X\ar[r]^{}&\Coker(f)\ar[r] &0\\
0\ar[r]^{}& \Im(f) \ar[u]^{h}\ar[r]_{i}&B \ar[u]^{x}\ar[r]_{}&\Coker(f)\ar[u]^{\cong} \ar[r] &0,\\}$$
 By (2), there exists $s_2\in S$ and an $R$-homomorphism $w:X\rightarrow E$ such that $w\circ v=s_2\Id_E$. So $w\circ x\circ f=w\circ x\circ i\circ f=w\circ v\circ h\circ f=s_1s_2\Id_E$. Hence  $0\rightarrow E\rightarrow B\rightarrow C\rightarrow 0$ is $u$-$S$-split with respect to $s_1s_2$.
 \end{proof}

\begin{corollary}\label{s-inj-char-1} Let $R$ be a ring, $S$ a multiplicative subset of $R$ and $E$  an $R$-module. Then the following two statements hold:
\begin{enumerate}
\item  If $E$ is $u$-$S$-injective, then  any $u$-$S$-short exact sequence  $0\rightarrow E\rightarrow B\rightarrow C\rightarrow 0$ beginning at $E$ is $u$-$S$-split;
\item If there exists $s\in S$ such that any short exact sequence  $0\rightarrow E\rightarrow B\rightarrow C\rightarrow 0$ beginning at $E$ is $u$-$S$-split with respect to $s$, then $E$ is $u$-$S$-injective.
\end{enumerate}
\end{corollary}
\begin{proof}

(1) Let $0\rightarrow E\xrightarrow{f} B\xrightarrow{g} C\rightarrow 0$ be a $u$-$S$-short exact sequence of $R$-modules. Set $h=\Id_E:E\rightarrow E$ be the identity map of $E$. Then  there exists $s\in S$ and $g:B\rightarrow E$ such that $s\Id_E=g\circ f$. Hence  $0\rightarrow E\rightarrow B\rightarrow C\rightarrow 0$ is $u$-$S$-split.

$(2)$ Let $f: A\rightarrow B$ be a  $u$-$S$-monomorphism. Then there is $s_1\in S$ such that $s_1\Ker(f)=0$.    Let $g:A\rightarrow E$ be an $R$-homomorphism. Consider the following push-out:
 $$\xymatrix@R=20pt@C=25pt{
E \ar[r]^{h}&X\ar[r]^{}&\Coker(h)\ar[r] &0\\
A \ar[u]^{g}\ar[r]_{f}&B \ar[u]^{l}\ar[r]_{\pi}&\Coker(f)\ar[u]^{\cong} \ar[r] &0,\\}$$
we have $\Ker(h)$ can be seen as a quotient of $\Ker(f)$. So $s_1\Ker(h)=0$. Hence $0\rightarrow E\rightarrow X\rightarrow \Coker(h)\rightarrow 0$ is $u$-$S$-exact, and thus $u$-$S$-split with respect to $s$. So there is an $R$-homomorphism $h':X\rightarrow E$
such that $h'\circ h=s\Id_E$ . Hence $h' \circ l\circ f=h' \circ h\circ g=sg$. Note that $s$ is independent with $g$. So $E$ is $u$-$S$-injective.
 \end{proof}

Recall that an $R$-module $P$ is said to be projective provided that the induced sequence $0\rightarrow \Hom_R(P,A)\rightarrow \Hom_R(P,B)\rightarrow \Hom_R(P,C)\rightarrow 0$ is exact for any exact sequence $0\rightarrow A\rightarrow B\rightarrow C\rightarrow 0$. Now we give a uniformly $S$-analogue of  projective modules.

\begin{definition}\label{s-projective} Let $R$ be a ring and $S$ a multiplicative subset of $R$. An $R$-module $P$ is called $u$-$S$-projective provided that the induced sequence $$0\rightarrow \Hom_R(P,A)\rightarrow \Hom_R(P,B)\rightarrow \Hom_R(P,C)\rightarrow 0$$ is $u$-$S$-exact for any $u$-$S$-exact sequence $0\rightarrow A\rightarrow B\rightarrow C\rightarrow 0$.
\end{definition}

%\begin{lemma}\label{u-S-tor-ext} Let  $R$ be a ring and $S$ a multiplicative subset of $R$. If $T$ is a $u$-$S$-torsion module, then $\Ext_R^{n}(T,M)$ and $\Ext_R^{n}(M,T)$ are $u$-$S$-torsion for any $R$-module $M$ and any $n\geq 0$.
%\end{lemma}

In common with the classical cases, we have the following characterizations of $u$-$S$-projective modules. Since the proof is very similar with that of characterizations of $u$-$S$-injective modules (see Proposition \ref{s-inj-exac} and \cite[Theorem 4.3]{QKWCZ21}), we omit the proof.

\begin{theorem}\label{s-proj-char}
Let $R$ be a ring, $S$ a multiplicative subset of $R$ and $P$ an $R$-module. Then the following statements are equivalent:
\begin{enumerate}
\item  $P$ is  $u$-$S$-projective;

\item for any $u$-$S$-epimorphism $B\xrightarrow{g} C$ there exists $s\in S$ such that for any $R$-homomorphism $h: P\rightarrow C$, there exists an $R$-homomorphism $\alpha: P\rightarrow B$ satisfying $sh=g\circ \alpha$;

\item for any short exact sequence $0\rightarrow A\xrightarrow{f} B\xrightarrow{g} C\rightarrow 0$, the induced sequence $0\rightarrow \Hom_R(P,A)\xrightarrow{f_\ast} \Hom_R(P,B)\xrightarrow{g_\ast} \Hom_R(P,C)\rightarrow 0$ is  $u$-$S$-exact;

\item  $\Ext_R^1(P,M)$ is  $u$-$S$-torsion for any  $R$-module $M$;

\item  $\Ext_R^n(P,M)$ is  $u$-$S$-torsion for any  $R$-module $M$ and $n\geq 1$.
\end{enumerate}
\end{theorem}

Similar to the proof of Corollary \ref{s-inj-char-1}, we have the following result.
\begin{corollary}\label{s-proj-char-1}
Let $R$ be a ring, $S$ a multiplicative subset of $R$ and $P$ an $R$-module. Then the following statements hold:
\begin{enumerate}
\item If  $P$ is  $u$-$S$-projective, then any $u$-$S$-short exact sequence $0\rightarrow A\xrightarrow{f} B\xrightarrow{g} P\rightarrow 0$ is $u$-$S$-split;

\item If there is $s\in S$ such that any short exact sequence $0\rightarrow A\xrightarrow{f} B\xrightarrow{g} P\rightarrow 0$ is $u$-$S$-split with respect to $s$, then $P$ is  $u$-$S$-projective.
\end{enumerate}
\end{corollary}

By Theorem \ref{s-proj-char}, projective modules are $u$-$S$-projective. Moreover, $u$-$S$-torsion modules are $u$-$S$-projective by \cite[Lemma 4.2]{QKWCZ21}.

\begin{corollary}\label{inj-ust-s-inj}
Let $R$ be a ring and $S$ a multiplicative subset of $R$. Let $P$ be a $u$-$S$-torsion $R$-module or a projective $R$-module. Then $P$ is $u$-$S$-projective.
\end{corollary}

\begin{proposition}\label{dird-sp} Let $R = R_1 \times R_2$ be direct product of rings $R_1$ and $R_2$, $S = S_1\times S_2 := \{(s_1, s_2) | s_1\in S_1, s_2\in S_2\}$ a direct product of multiplicative subsets of $R_1$ and $R_2$. Set $e_1 = (1, 0)$ and $e_2 = (0, 1)$. Then $P$ is a $u$-$S$-projective $R$-module if and only if $Pe_i$ is a $u$-$S_i$-projective $R_i$-module for each $i = 1, 2$.
\end{proposition}
\begin{proof}
Suppose $P$ is $u$-$S$-projective. Then $P\cong P e_1\times P e_2$. Let $N$ be an $R_1$-
module. Then, as $R_1$-modules, we have $\Ext^1_R(M, N \times 0)\cong \Ext^1_{R_1} (P e_1, N)$ which
is $u$-$S_1$-torsion. Consequently, $P e_1$ is a $u$-$S_1$-projective $R_1$-module. Similarly, $P e_2$ is
a $u$-$S_2$-projective $R_2$-module. On the other hand, suppose $P e_i$
is a $u$-$S_i$-projective
$R_i$-module for each $i = 1, 2$. Let $N$ be an $R$-module. Then $N\cong Ne_1\times Ne_2$.
So $\Ext^1_R(P, N) \cong \Ext^1_{R_1} (P e_1, Ne_1) \times \Ext^1_{R_2} (P e_2, Ne_2)$ which is $u$-$S$-torsion. Consequently, $P$ is a $u$-$S$-projective $R$-module.
\end{proof}

Recall from \cite{z21} that an $R$-module $F$ is  $u$-$S$-flat if and only if $\Tor^R_1(M,F)$ is  $u$-$S$-torsion for any  $R$-module $M$.

\begin{proposition}\label{S-proj-flat} Let  $R$ be a ring and $S$ a multiplicative subset of $R$. If $P$ is a $u$-$S$-projective $R$-module, then $P$ is $u$-$S$-flat.
\end{proposition}
\begin{proof} Let $P$ be a $u$-$S$-projective $R$-module, $M$ an $R$-module and $E$  an injective cogenerator. Then there is an element $s\in S$ such that $s\Ext_R^1(P,\Hom_R(M,E))=0$ by Theorem \ref{s-proj-char}. By \cite[Lemma 2.16(b)]{gt}, we have $s\Hom_R(\Tor_1^R(P,M),E))=0$. Let $f: s\Tor_1^R(P,M)\rightarrow E$ be an $R$-homomorphism. Since $E$ is injective,  there is an $R$-homomorphism $g: \Tor_1^R(P,M)\rightarrow E$ such that $f=g\circ i$ where $i: s\Tor_1^R(P,M)\hookrightarrow \Tor_1^R(P,M)$ is the embedding map. Since $s\Hom_R(\Tor_1^R(P,M),E))=0$, we have  $f(sx)=g(sx)=sg(x)=0$ for any $x\in \Tor_1^R(P,M)$. Thus  $\Hom_R(s\Tor_1^R(P,M),E)=0$. So  $s\Tor_1^R(P,M)=0$ since $E$ is an injective cogenerator. Consequently,  $P$ is $u$-$S$-flat.
%Let $p:F \twoheadrightarrow P$ be an epimorphism with $F$ flat and let $q:P \rightarrow F$ such that $p\circ q=s\Id_P$ with $s\in S$.Then $p,q$ induce the maps $p': \Tor^R_1(F,M)\rightarrow \Tor^R_1(P,M)$ and $q': \Tor^R_1(P,M)\rightarrow \Tor^R_1(F,M)$ with $p'\circ q'=s\Id_{\Tor^R_1(P,M)}$. As $\Tor^R_1(F,M)=0$, we get $s\Tor^R_1(P,M)=0$.
\end{proof}

\begin{proposition}\label{s-proj-exac}
Let $R$ be a ring and $S$ a multiplicative subset of $R$. Then the following statements hold.
\begin{enumerate}
\item Any finite direct sum of  $u$-$S$-projective modules is  $u$-$S$-projectve.
\item Let $0\rightarrow A\xrightarrow{f} B\xrightarrow{g} C\rightarrow 0$ be a $u$-$S$-exact sequence. If $C$ is  $u$-$S$-projective, then $A$ is  $u$-$S$-projective if and only if so is $B$.
\item   Let  $A\rightarrow B$ be a $u$-$S$-isomorphism. Then $A$  is  $u$-$S$-projective if and only if $B$ is  $u$-$S$-projective.
\item  Let  $0\rightarrow A\xrightarrow{f} B\xrightarrow{g} C\rightarrow 0$  be a $u$-$S$-split $u$-$S$-exact sequence.  If $B$ is  $u$-$S$-projective, then $A$  and $C$ are  $u$-$S$-projective.
\end{enumerate}
\end{proposition}
\begin{proof}
We only prove $(4)$ since the proof of $(1)$-$(3)$ is dual to that of \cite[Proposition 4.7]{QKWCZ21}.

$(4)$:  Let $X$ be a module. We will prove that $\Ext_R^1(C,X)$ is annihilated by some element of $S$.
Let  $h:C \rightarrow B$ such that $g\circ h = s\Id_C$ with $s \in S$.
Then $g,h$ induce the maps $g': \Ext_R^1(C,X) \rightarrow \Ext_R^1(B,X)$ and $h': \Ext_R^1(B,X) \rightarrow \Ext_R^1(C,X)$  with  $h'\circ g'=s\Id_{\Ext_R^1(C,X)}$.
As $t\Ext_R^1(B,X) = 0$ for some $ t \in S$, we get $st\Ext_R^1(C,X) = 0$.
The ``$A$-part'' of the proof goes similarly.

%Let  $0\rightarrow A\xrightarrow{f} B\xrightarrow{g} C\rightarrow 0$  be a $u$-$S$-split $u$-$S$-exact sequence where $B$ is  $u$-$S$-projective. Let  $h: M\rightarrow N$ be a $u$-$S$-epimorphism. Let $\beta: C\rightarrow N$ be an $R$-homomorphism. Consider the following diagram where both rows are $u$-$S$-exact.
%$$\xymatrix@R=20pt@C=25pt{
%B \ar[r]^{g}\ar@{-->}[d]^{\alpha}&C\ar[r] \ar[d]^{\beta} &0\\
%M\ar[r]^{h}&N\ar[r] &0\\}$$
%There is an $R$-homomorphism $\alpha:B\rightarrow M$ and $s\in S$ such that  $s\beta\circ g=h\circ \alpha$. Since $g$ is a $u$-$S$-split $u$-$S$-epimorphism. Then there is an $R$-homomorphism $g':C\rightarrow B$ such that $g\circ g'=t\Id_C$ for some $t\in S$. So $s\beta=s\beta\circ g\circ g'=h\circ \alpha\circ g'$. So $C$ is $u$-$S$-projective. Let $\gamma: A\rightarrow N$ be an $R$-homomorphism. Consider the following diagram where both rows are $u$-$S$-exact.
%$$\xymatrix@R=20pt@C=25pt{
%0\ar[r]&A \ar[r]^{g}\ar[rd]^{\gamma}&B  &\\
%&M\ar[r]^{h}&N\ar[r] &0\\}$$
%Let $g': B\rightarrow A$ be an $R$-homomorphism such that $g'\circ g=s\Id_A$ for some $s\in S$. Since $B$ is $u$-$S$-projective, there exists an $R$-homomorphism $\delta:B\rightarrow M$ such that $h\circ\delta=s'\gamma\circ g'$ for some $s'\in S$.
%Thus $h\circ\delta\circ g=s'\gamma\circ g'\circ g=ss'\gamma$. So $A$  is $u$-$S$-projective.
\end{proof}

It is well-known that any direct sum of  projective modules is projective. However, the following example shows that a direct sum of u-S-projective modules is not necessarily $u$-$S$-projective.

\begin{example}
Let $R=\mathbb{Z}$ be the ring of integers, $p$ a prime in $\mathbb{Z}$ and  $S=\{p^n|n\in\mathbb{N}\}$. Let $M_n= \mathbb{Z}/\langle p^n\rangle$ for each $n\geq 1$. Then $M_n$ is $u$-$S$-torsion and thus  $u$-$S$-projective. Set $N=\bigoplus\limits_{n=1}^{\infty}M_n$. Note that $\Ext_{\mathbb{Z}}^1(\mathbb{Z}/\langle p^n\rangle,\mathbb{Z}/\langle p^m\rangle)\cong \mathbb{Z}/\langle p^{\min\{m,n\}}\rangle$. We have  $\Ext_{\mathbb{Z}}^1(N,N)\cong \prod\limits_{n\in\mathbb{N}}(\bigoplus\limits_{m\in\mathbb{N}}\Ext_{\mathbb{Z}}^1(\mathbb{Z}/\langle p^n\rangle,\mathbb{Z}/\langle p^m\rangle)))\cong \prod\limits_{n\in\mathbb{N}}(\bigoplus\limits_{m\in\mathbb{N}} \mathbb{Z}/\langle p^{\min\{m,n\}}\rangle)$. Note that the abelian group $\prod\limits_{n\in\mathbb{N}}(\bigoplus\limits_{m\in\mathbb{N}} \mathbb{Z}/\langle p^{\min\{m,n\}}\rangle)$ contains a subgroup $\prod\limits_{n\in\mathbb{N}}\mathbb{Z}/\langle p^{n}\rangle $. Since $\prod\limits_{n\in\mathbb{N}}\mathbb{Z}/\langle p^{n}\rangle $ is not $u$-$S$-torsion, we have $\Ext_{\mathbb{Z}}^1(N,N)$ is also not $u$-$S$-torsion.  Consequently $N$ is not $u$-$S$-projective.
\end{example}

Let $\p$ be a prime ideal of $R$. We say an $R$-module $P$ is (simply) \emph{$u$-$\p$-projective} provided that  $P$ is  $u$-$(R\setminus\p)$-projective.
\begin{proposition}\label{s-flat-loc-char}
Let $R$ be a ring and $P$ an $R$-module. Then the following statements are equivalent:
 \begin{enumerate}
\item  $P$ is projective;
\item   $P$ is    $u$-$\p$-projective for any $\p\in \Spec(R)$;
\item   $P$ is   $u$-$\m$-projective for any $\m\in \Max(R)$.
 \end{enumerate}
\end{proposition}
\begin{proof} $(1)\Rightarrow (2)\Rightarrow (3):$  Trivial.

 $(3)\Rightarrow (1):$ Let $M$ be an $R$-module. Then $\Ext^1_R(P,M)$ is $(R\setminus\m)$-torsion. Thus for any $\m\in \Max(R)$, there exists  $s_{\m}\in R\setminus\m$ such that $s_{\m}\Ext^1_R(P,M)=0$. Since the ideal generated by  $\{s_{\m}\mid \m\in \Max(R)\}$ is $R$, we have $\Ext^1_R(P,M)=0$. So $P$ is projective.
\end{proof}

\section{$u$-$S$-semisimple modules and $u$-$S$-semisimple rings}

%\begin{definition}
%Let $R$ be a ring and  $S$ a multiplicative subset of $R$ and $M$ a non-$u$-$S$-torsion module. An $R$-module $M$ is called $u$-$S$-simple provided that for any $u$-$S$-monomorphism $ N\xrightarrow{i} M$ we have $N$ is $u$-$S$-torsion module or $i$ is a $u$-$S$-isomorphism.
%\end{definition}

Let $R$ be a ring. Recall from \cite{R09} that an $R$-module $M$ is semisimple provided that it is a direct sum of simple modules. By \cite[Proposition 4.1]{R09} an $R$-module $M$ is semisimple if and only if every submodule is a direct summand of $M$. So $M$ is semisimple if and only if any short exact sequence $0\rightarrow A\rightarrow M\rightarrow C\rightarrow 0$ is split. Utilizing this characterization,  we introduce the notion of $u$-$S$-semisimple module.

\begin{definition}
Let $R$ be a ring and  $S$ a multiplicative subset of $R$. An $R$-module $M$ is called $u$-$S$-semisimple provided that any $u$-$S$-short exact sequence $0\rightarrow A\rightarrow M\rightarrow C\rightarrow 0$ is $u$-$S$-split.
\end{definition}

Obviously, $u$-$S$-torsion modules are $u$-$S$-semisimple. Certainly, the class of $u$-$S$-semisimple modules is closed under $u$-$S$-isomorphisms. We can deduce that semisimple modules are also $u$-$S$-semisimple from the following lemma.
\begin{lemma}\label{s-semisimple-exac} An $R$-module $M$ is $u$-$S$-semisimple if and only if  any short exact sequence $0\rightarrow L\rightarrow M\rightarrow N\rightarrow 0$ is $u$-$S$-split.
\end{lemma}
\begin{proof}  The ``only if part'' is clear. To prove the converse, let  $0\rightarrow A\xrightarrow{f} M\xrightarrow{g} C\rightarrow 0$ be a $u$-$S$-short exact sequence. Consider the natural exact sequence $0\rightarrow \Ker(g)\rightarrow M\xrightarrow{g_1} \Im(g)\rightarrow 0$. Then there is an $R$-homomorphism $g_1':\Im(g)\rightarrow M$ and $s\in S$ such that $g_1\circ g_1'=s\Id_{\Im(g)}$. Let $i:\Im(g)\rightarrow C$ be the embedding map. Then by Lemma \ref{l}, there exists a $u$-$S$-isomorphism $j:C\rightarrow \Im(g)$ such that $i\circ j=s'\Id_C$ for some $s'\in S$. Setting $g'= g_1'\circ j$, we have  $g\circ g'=ss'\Id_{C}$. So  the $u$-$S$-short exact sequence  $0\rightarrow A\xrightarrow{f} M\xrightarrow{g} C\rightarrow 0$ is $u$-$S$-split by Lemma \ref{s-split}.
\end{proof}

\begin{proposition}\label{s-semisimple-s-exact} Let  $0\rightarrow A\xrightarrow{f} B\xrightarrow{g} C\rightarrow 0$  be a $u$-$S$-short exact sequence. If $B$ is $u$-$S$-semisimple, then $A$ and $C$ are $u$-$S$-semisimple.
\end{proposition}
\begin{proof}

Let $a:X\rightarrow A$ be a $u$-$S$-monomorphism. Since $f$ and $a$ are $u$-$S$-monomorphisms, so is their composition  $f\circ a: X \rightarrow B$. Indeed, if $t\Ker(f)=0$ and $t'\Ker(a)=0$ with $t,t'\in S$, then $tt'\Ker(f\circ a)=0$. As $B$ is $u$-$S$-semisimple, $w\circ f\circ  a=s\Id_X$ for some  $R$-homomorphism $w:B\rightarrow X$ and some $s\in S$. So $0 \rightarrow X \rightarrow A\rightarrow  Y\rightarrow 0$ $u$-$S$-splits, where $Y=\Coker(a)$. Similarly, let $j:C\rightarrow K$ be a $u$-$S$-epimorphism. As $g$ and $j$ are $u$-$S$-epimorphisms, so is their composition  $j\circ g: B\rightarrow  K$. Indeed, if $tK\subseteq \Im(j)$  and $t'C\subseteq \Im(g)$  with $t,t'\in S$, then $tt'K\subseteq \Im(j\circ g)$. As $B$ is $u$-$S$-semisimple, $j\circ g\circ w=s\Id_K$ for some linear map $w:K\rightarrow B$ and some $s\in S$. Let $M=\Ker(j)$, then $0\rightarrow M\rightarrow C\rightarrow K \rightarrow 0$  $u$-$S$-splits by Lemma \ref{s-split}.

\end{proof}

Recall that  a ring $R$ is semisimple provided that $R$ is semisimple as an $R$-module. Note that a ring $R$ is semisimple if and only if any free  $R$-module is semisimple by \cite[Proposition 4.5]{R09}. To give a ``uniform'' version of semisimple rings, we define $u$-$S$-semisimple rings by considering all free  $R$-modules.
\begin{definition}
Let $R$ be a ring and  $S$ a multiplicative subset of $R$. $R$ is called a $u$-$S$-semisimple ring provided that any free  $R$-module is $u$-$S$-semisimple.
\end{definition}

Obviously, all semisimple rings are $u$-$S$-semisimple for any multiplicative subset $S$ of $R$. The next result gives various characterizations of $u$-$S$-semisimple rings.

\begin{theorem}\label{s-semisimple-char-}
Let $R$ be a ring and $S$ a multiplicative subset of $R$. Then the following statements are equivalent:
\begin{enumerate}
\item $R$ is a $u$-$S$-semisimple ring;
\item any $R$-module is $u$-$S$-semisimple;
\item any $u$-$S$-short exact sequence is $u$-$S$-split;
\item any short exact sequence is $u$-$S$-split;
\item $\Ext_R^1(M,N)=0$ is $u$-$S$-torsion for any $R$-modules $M$ and $N$;
\item any $R$-module is $u$-$S$-projective;
\item  any $R$-module is $u$-$S$-injective.
\end{enumerate}
\end{theorem}
\begin{proof} $(1)\Rightarrow (2)$:  Let $M$ be an $R$-module. There exists an exact sequence $0\rightarrow K\rightarrow F\rightarrow M\rightarrow 0$ with $F$ free $R$-module. By Proposition \ref{s-semisimple-s-exact}, $M$ is $u$-$S$-semisimple.

$(2)\Rightarrow (3)$:  Let $\xi: 0\rightarrow A\rightarrow B\rightarrow C\rightarrow 0$ be a $u$-$S$-short exact sequence. Since $B$ is  $u$-$S$-semisimple, the  $u$-$S$-short exact sequence $\xi$ is $u$-$S$-split.

$(3)\Rightarrow (2)$:  Let $M$ be an $R$-module and $0\rightarrow A\rightarrow M\rightarrow B\rightarrow 0$  a $u$-$S$-short exact sequence.  By $(3)$,  $0\rightarrow A\rightarrow M\rightarrow B\rightarrow 0$  is $u$-$S$-split. So $M$ is $u$-$S$-semisimple.

$(2)\Rightarrow (1)$ and $(3)\Rightarrow (4)$: Trivial.

$(4)\Rightarrow (6)$: Let $M$ be an $R$-module and $0\rightarrow K\rightarrow P\rightarrow M\rightarrow 0$ be a  short exact sequence with $P$ projective. Then $M$ is $u$-$S$-projective by Proposition \ref{s-proj-exac}.

$(5)\Leftrightarrow (6)$: This equivalence follows from Theorem 2.7.

$(5)\Leftrightarrow (7)$: This equivalence follows from \cite[Theorem 4.3]{QKWCZ21}.

$(6)\Rightarrow (3)$: Let $0\rightarrow N\rightarrow K\rightarrow M\rightarrow 0$ be a $u$-$S$-short exact sequence. Since $M$ is a $u$-$S$-projective module, then  $0\rightarrow N\rightarrow K\rightarrow M\rightarrow 0$ is $u$-$S$-split by Corollary \ref{s-proj-char-1}.
\end{proof}

\begin{corollary}\label{s-semisimple-char}
Let $R$ be a ring and $S$ a multiplicative subset of $R$. Suppose $R$ is a $u$-$S$-semisimple ring. Then $R$ is both $u$-$S$-Noetherian and  $u$-$S$-von Neumann regular. Consequently, there exists an element $s\in S$ such that for any ideal $I$ of $R$ there is an $R$-homomorphism $f_I: R\rightarrow I$ satisfying $f_I(i)=si$ for any $i\in I$.
\end{corollary}
\begin{proof}
 Let $\Gamma:=\{I\}_{I\unlhd R}$ be the set of all ideals of $R$. Considering the natural short exact sequence  $0\rightarrow \bigoplus\limits_{I\in \Gamma} I\xrightarrow{i} \bigoplus\limits_{\Gamma} R\xrightarrow{\pi} \bigoplus\limits_{I\in \Gamma} R/I \rightarrow 0$, we have an $R$-homomorphism $i':  \bigoplus\limits_{\Gamma} R\rightarrow \bigoplus\limits_{I\in \Gamma} I$ such that $i'\circ i=s\Id_{\bigoplus\limits_{I\in \Gamma} I}$ for some $s\in S$.  So the natural embedding map $\Im(i')\hookrightarrow \bigoplus\limits_{I\in \Gamma} I$ is a $u$-$S$-isomorphism. Thus the set $\Gamma:=\{I\}_{I\unlhd R}$ is uniformly $S$-finite (see \cite{QKWCZ21} for example) since the $I$-th component of  $\Im(i')$ is finitely generated for any ideal $I$ of $ R$. So $R$ is a uniformly $S$-Noetherian ring. Since any $R$-module is $u$-$S$-projective by $(4)\Rightarrow (6)$ of Theorem \ref{s-semisimple-char-}, we have $R$ is uniformly $S$-von Neumann regular by Proposition \ref{S-proj-flat}.

Let $\Gamma:=\{I\}_{I\unlhd R}$ be the set of all ideals of $R$. Since $R$ is uniformly $S$-Noetherian, there exists an element $s\in S$ such that for any ideal $I\in \Gamma$ there is a finitely generated sub-ideal $K$ of $I$ satisfying $sI\subseteq K$.  Since $R$ is uniformly $S$-von Neumann regular, there is an  element  $s'\in S$ such that for any finitely generated ideal $K$ of $R$ there is an idempotent $e\in K$ such that $s'(K/\langle e\rangle)=0$  by \cite[Theorem 3.13]{z21}. Let $f: R\rightarrow I$ be the $R$-homomorphism given by $f(1) = ss'e$. Then we have $f(i)=ss'i$ for any $i\in I$.
\end{proof}

Certainly, if $R$ is a $u$-$S$-semisimple ring, then $R$ is $u$-$S$-semisimple as an $R$-module. However, the following example shows that the converse does not hold in general.
\begin{example} Let $R=\mathbb{Z}$ be the ring of all integers and the multiplicative subset $S=\mathbb{Z}\setminus\{0\}$. Let $\langle n\rangle$ be the ideal  generated by $n\in\mathbb{Z}$, and consider the exact sequence $0\rightarrow \langle n\rangle\xrightarrow{i} \mathbb{Z} \rightarrow \mathbb{Z}/ \langle n\rangle\rightarrow 0$. Set $i':\mathbb{Z}\rightarrow \langle n\rangle$ to be the $\mathbb{Z}$-homomorphism satisfying $i'(1)=n$. Then $i'(i(m))=nm$ for any $m\in \langle n\rangle$. Thus $\mathbb{Z}$ is a $u$-$S$-semisimple $\mathbb{Z}$-module. Since $R$ is not uniformly $S$-von Neumann regular by \cite[Example 3.15]{z21},  $R$ is not a $u$-$S$-semisimple ring by Corollary \ref{s-semisimple-char}.
\end{example}

Moreover, the following result shows that any $u$-$S$-semisimple ring is in fact a semisimple ring in the case that $S$ is regular multiplicative subset of $R$, i.e., the multiplicative set $S$ is composed of  non-zero-divisors.

\begin{proposition}\label{s-vn-vn} Let $R$ be a ring and $S$ a regular multiplicative subset of $R$.
 Then $R$ is a $u$-$S$-semisimple ring if and only if  $R$ is  a semisimple ring.
\end{proposition}
\begin{proof} The ``if part'' is clear. To prove the converse, let  $R$ be a $u$-$S$-semisimple ring, where each element in $S$ is a non-zero-divisor. There exists an element $s\in S$ such that for any ideal $I$ of $R$ there is an $R$-homomorphism $f: R\rightarrow I$ satisfying $f(i)=si$ for any $i\in I$ by Corollary \ref{s-semisimple-char}.  Set $I=\langle s^2\rangle$. Then $s^2f(1)=f(s^2)=s^3$. Let $f(1)=s^2r\in I$ for some $r\in R$. Then $s^4r=s^3$. Since $s$ is a non-zero-divisor, we have $sr=1$ and thus $s$ is a unit. Consequently,  $R$ is  a semisimple ring.
\end{proof}

Suppose $R$ is  $u$-$S$-semisimple as an $R$-module for a regular multiplicative subset $S$ of $R$. We must note that  $R$ is not necessarily a semisimple ring.

\begin{example} Let $R$ be a non-field domain and $S=R\setminus\{0\}$ the set of all nonzero elements in $R$. Then $R$ is obviously not a semisimple ring. However, $R$ is $u$-$S$-semisimple as an $R$-module. Indeed, let $I$ be an nonzero ideal of $R$ and $0\not=s\in I$. Let $f: R\rightarrow I$ be an $R$-homomorphism satisfying $f(1)=s$. Then we have $f(i)=si$ for any $i\in I$. Hence $R$ is $u$-$S$-semisimple as an $R$-module by Lemma \ref{s-semisimple-exac}.
\end{example}

We also have the following direct product property of $u$-$S$-semisimple rings.
\begin{proposition}\label{dird-s-semp} Let $R = R_1\times R_2$ be direct product of rings $R_1$ and $R_2$ and $S = S_1\times S_2$ a direct product of multiplicative subsets of $R_1$ and $R_2$. Then $R$ is a $u$-$S$-semisimple ring if and only if $R_i$ is a $u$-$S_i$-semisimple ring for each $i=1, 2$.
\end{proposition}
\begin{proof}
 Follows by Proposition \ref{dird-sp} and Theorem \ref{s-semisimple-char-}.
\end{proof}

The following non-trivial example shows that the condition that ``$S$ is a regular multiplicative subset
of $R$'' in Proposition \ref{s-vn-vn} cannot be removed.

\begin{example}\label{us-sim-not-sem} Let $R_1$ be a semi-simple ring and $R_2$ a non-semi-simple ring. Denote by $R = R_1\times R_2$. Then R is not a semi-simple ring. Set $S = \{(1, 1),(1, 0)\}$
which is a multiplicative subset of $R$. Then $R$ is trivially a $u$-$S$-semi-simple ring by
Proposition \ref{dird-s-semp}.
\end{example}

Let $\p$ be a prime ideal of $R$. We say a ring $R$ is (simply) a \emph{$u$-$\p$-semisimple ring} provided  $R$ is a $u$-$(R\setminus\p)$-semisimple ring. The final result gives a new characterization of semisimple rings.
\begin{proposition}\label{s-semi-loc-char}
Let $R$ be a ring. Then the following statements are equivalent:
 \begin{enumerate}
\item  $R$ is a semisimple ring;
\item   $R$ is a $u$-$\p$-semisimple ring for any $\p\in \Spec(R)$;
\item   $R$ is a $u$-$\m$-semisimple ring for any $\m\in \Max(R)$.
 \end{enumerate}
\end{proposition}
\begin{proof} $(1)\Rightarrow (2):$  Let $P$ be an $R$-module and $\p\in \Spec(R)$. Then $P$ is projective, and thus is $u$-$\p$-projective. So $R$ is a $u$-$\p$-semisimple ring by Theorem \ref{s-semisimple-char-}.

 $(2)\Rightarrow (3):$  Trivial.

 $(3)\Rightarrow (1):$  Let $M$ be an $R$-module. Then $M$ is  $u$-$\m$-projective for any $\m\in \Max(R)$. Thus $M$ is projective by Proposition \ref{s-flat-loc-char}. So $R$ is a semisimple ring.
\end{proof}

\begin{acknowledgement}\quad\\
The authors would like to thank the reviewers for many valuable suggestions.
The first author was supported by  the National Natural Science Foundation of China (No. 12061001).
\end{acknowledgement}


\begin{thebibliography}{99}
\bibitem{ad02}  D. D.  Anderson, T.  Dumitrescu, {\it S-Noetherian rings}, Commun. Algebra {\bf 30} (2002), 4407-4416.

\bibitem{dhz19}  D. D. Anderson, A. Hamed, M. Zafrullah, {\it On S-GCD domains}, J. Algebra Appl.
(2019), 1950067 (14 pages).



\bibitem{s19} S. Bazzoni,  L. Positselski,  {\it  S-almost perfect commutative rings}, J. Algebra {\bf 532} (2019), 323-356.

\bibitem{bh18} D. Bennis,  M. El Hajoui,  {\it  On S-coherence}, J. Korean Math. Soc. \textbf{55} (2018), no. 6, 1499-1512.


\bibitem{FS01} L. Fuchs, L. Salce,   {\it  Modules over Non-Noetherian Domains}, Providence, AMS, 2001.

\bibitem{g} S. Glaz,  {\it   Commutative Coherent Rings},  Lecture Notes in Mathematics, vol. \textbf{1371}, Berlin: Spring-Verlag, 1989.

\bibitem{gt}  R. Gobel, J. Trlifaj,  {\it Approximations and Endomorphism Algebras of Modules}, De Gruyter Exp. Math., vol.  {\bf 41}, Berlin: Walter de Gruyter GmbH \& Co. KG, 2012.

\bibitem{g86}  J. S. Golan,  {\it  Torsion Theories}, Pitman Monographs and Surveys in Pure and Applied Mathematics Series, vol. 29, New York: Longman Scientic and Technical, 1986.

\bibitem{kkl14} H. Kim,  M. O. Kim,  J. W.  Lim,  {\it  On S-strong Mori domains}, J. Algebra  {\bf 416} (2014),  314-332.

\bibitem{l15} J. W. Lim,  {\it  A Note on S-Noetherian Domains}, Kyungpook Math. J. {\bf 55} (2015), 507-514.

\bibitem{lO14} J. W. Lim, D. Y. Oh, {\it S-Noetherian properties on amalgamated algebras along an ideal},  J. Pure Appl. Algebra {\bf 218} (2014), 2099-2123.



\bibitem{N65} C. N\v{a}st\v{a}sescu, C. Nita,  {\it  Objects noeth\'{e}riens par rapport \`{a} une sous-categorie \'{e}paisse \'{d}un cat\'{e}gorie abelienne}, Rev. Roum. Math. Pures et Appl. {\bf9}  (1965), 1459-1468.

\bibitem{QKWCZ21} W. Qi,\  H. Kim,\ F. G. Wang,\ M. Z. Chen,\ W. Zhao, {\it  Uniformly S-Noetherian rings}, https://arxiv.org/abs/2201.07913.

\bibitem{R09} J. Rotman, {\it An Introduction to Homological Algebra}, Second edition, Universitext, New York: Springer, 2009.

\bibitem{S20} E. S. Sevim, U. Tekir, S. Koc, {\it S-Artinian rings and finitely S-cogenerated rings}, J.
Algebra Appl. (2020), 2050051 (16 pages).

\bibitem{S75}  B. Stenstr\"{o}m,  {\it   Rings of Quotients}, Die Grundlehren Der Mathematischen Wissenschaften, Berlin: Springer-verlag,  1975.

\bibitem{fk16} F. G.  Wang,  H. Kim,  {\it  Foundations of Commutative Rings and Their Modules}, Singapore:  Springer, 2016.

\bibitem{z21} X. L. Zhang, {\it Characterizing S-flat modules and S-von Neumann regular rings by uniformity}, Bull. Korean Math. Soc.,  {\bf 59} (2022), no. 3,   643-657.
\end{thebibliography}
\end{document}